\documentclass[12pt]{article}
\usepackage[a4paper]{geometry}
\usepackage{graphicx} % Required for inserting images
\usepackage{mathtools,amsmath,amssymb,amsthm,url,hyperref,color}
\newtheorem{thm}{Theorem}[section]

\newtheorem{lemma}[thm]{Lemma}

\newtheorem*{thm*}{Theorem}
\newtheorem*{lemma*}{Lemma}
\newtheorem*{prop*}{Proposition}
\newtheorem*{cor*}{Corollary}
\newtheorem*{conj*}{Conjecture}

\theoremstyle{definition}
\newtheorem{defn}[thm]{Definition}

\newtheorem{ex}[thm]{Example}

\newtheorem*{ques*}{Question}

\DeclareMathOperator{\free}{free}

\title{Local and global $d$-rigidity are not definable in the first order logic of graphs}
\author{Daniel Irving Bernstein and Nathaniel Vaduthala}
\date{\today}

\begin{document}

\maketitle

\begin{abstract}
    We use Hanf locality and a result of Cruickshank, Jackson, and Tanigawa on the global rigidity of graphs of $k$-circuits
    to prove that local and global $d$-rigidity are not definable in the first order logic of graphs.
\end{abstract}

\section{Introduction}
One can imagine constructing a graph $G$ in $d$-dimensional space using rigid bars as edges that are free to move around their incident vertices,
and ask if the resulting structure is rigid.
If the positions of the vertices of such a structure are sufficiently ``generic,'' then whether the resulting structure is rigid depends only on $G$ and $d$, and not how the graph is built; in this case, we say that $G$ is \emph{locally $d$-rigid}.
There is also a related (stronger) notion of \emph{global $d$-rigidity}.
The search for elegant combinatorial descriptions of the classes of locally and globally $d$-rigid graphs has motivated a lot of research in rigidity theory.
Such characterizations have been obtained for $d = 1$ and $d = 2$~\cite{pollaczek1927gliederung,connelly2005generic}.
In spite of over a century of effort, the $d \ge 3$ case remains open (but see~\cite{clinch2019abstract} for promising related work on the $d=3$ case).

This raises the question: how sophisticated of a language does one need to characterize rigidity in $d \ge 3$ dimensions?
Mathematical logic provides the concepts necessary to make this question precise and to start looking for answers.
Our hope is that this paper will be the first of many that make connections between rigidity theory and mathematical logic
for the purposes of lower-bounding the complexity of the theoretical framework required to combinatorially characterize rigidity.
To this end, the main result of our paper is Theorem~\ref{thm: not first-order},
which says that any combinatorial characterization of local or global $d$-rigidity cannot be phrased in first-order logic.
This generalizes the well-known result that graph connectivity is not definable in first order logic.

Knowing that a graph property is expressible in a particular logic has complexity implications.
Perhaps the most famous graph-theoretic instance of this is a theorem of Courcelle~\cite{Courcelle_graph},
which says that if a graph property is expressible in \emph{(counting) monadic second order logic},
then the decision problem of whether a given graph satisfies that property can be solved in linear time,
if one restricts to graphs of bounded treewidth.

The paper is structured as follows.
Section~\ref{sec: rigidity theory} defines local and global $d$-rigidity and states what is known about them.
Section~\ref{sec: logic} has the necessary background from mathematical logic.
This includes the first-order logic definitions of formulas, sentences, theories, models, and queries,
as well as Hanf locality, which is the tool from model theory that we use in Section~\ref{sec: the construction} to prove that local and global $d$-rigidity are not expressible in first-order logic.
Our application of Hanf locality requires us to construct a pair of graphs $G$ and $H$ satisfying a certain ``local isomorphism'' condition such that $G$ is globally rigid but $H$ is not.
Our construction here makes crucial use of recent work of Cruickshank, Jackson, and Tanigawa~\cite{cruickshank2024global} on global rigidity of triangulated manifolds.

\section{Background on rigidity theory}\label{sec: rigidity theory}
A \emph{$d$-framework} is a pair $(G,p)$ consisting of a graph $G = (V,E)$ and a function $p: V \rightarrow \mathbb{R}^d$.
One should think of $d$-frameworks as physical constructions of graphs in $d$-dimensional space where the edges are rigid bars that are free to move around their incident vertices; we will be interested in whether such structures are \emph{rigid}, a notion we now make precise.
Two $d$-frameworks $(G,p)$ and $(G,q)$ are \emph{equivalent} if
\begin{equation}\label{eq: same length}
    \|p(v)-p(w)\| = \|q(v)-q(w)\|
\end{equation}
whenever $uv$ is an edge of $G$; they are moreover \emph{congruent} when \eqref{eq: same length} holds for every pair $u,v \in V$ (not just edges).
It follows from the definitions that two congruent frameworks are equivalent, but the converse need not hold.
A $d$-framework $(G,p)$ is said to be \emph{globally rigid} if every equivalent framework is congruent,
and \emph{locally rigid} if there exists $\varepsilon > 0$ such that $(G,q)$ is congruent to $(G,p)$
whenever $\|p(v) - q(v)\| \le \varepsilon$ for all $v \in V$ and $(G,q)$ is equivalent to $(G,p)$.

For any graph $G = (V,E)$, it is either the case that almost every (with respect to Lebesgue measure) $p: V \rightarrow \mathbb{R}^d$ satisfies the property that
$(G,p)$ is locally/globally rigid, or that almost every $p: V \rightarrow \mathbb{R}^d$ satisfies the property that
$(G,p)$ is \emph{not} locally/globally rigid~\cite{asimow1978rigidity,gortler2010characterizing}.
In other words, for sufficiently generic $p: V \rightarrow \mathbb{R}^d$, local/global rigidity of the framework $(G,p)$ depends only on $G$ and not on $p$.
With this in mind, the following definition is well-formulated.

\begin{defn}
    A graph $G = (V,E)$ is \emph{globally/locally $d$-rigid} if $(G,p)$ is globally/locally $d$-rigid for almost every $p: V \rightarrow \mathbb{R}^d$.
\end{defn}

The search for elegant combinatorial characterizations of the classes of globally/locally $d$-rigid graphs for each $d$ is a major part of rigidity theory research.
It is a relatively straightforward exercise to check that a graph is locally $1$-rigid if and only if it is connected.
A graph is globally $1$-rigid if and only if its $2$-connected~\cite[Theorem~3.8]{jackson2007notes}.
A characterization of local $2$-rigidity due to Hilda Geiringer is as follows.
\begin{thm}{\cite{pollaczek1927gliederung}}\label{thm: local 2-rigidity}
    Let $G = (V,E)$ be a graph. Then $G$ is locally $2$-rigid if and only if $G$ has a spanning subgraph $H$ satisfying the following properties
    \begin{enumerate}
        \item $H$ has $2|V| - 3$ edges, and
        \item every subgraph of $H$ on $m$ vertices has at most $2m-3$ edges.
    \end{enumerate}
\end{thm}
See~\cite{connelly2005generic} for a characterization of global $2$-rigidity.
Combinatorial characterizations of local and global $d$-rigidity for $d\ge 3$ remain elusive.

\section{Background on mathematical logic and model theory}\label{sec: logic}
We now provide the necessary background on first-order logic and model theory.
For a more comprehensive introduction, we refer to \cite{ebbinghaus_mathlogic}.

\subsection{Structures and First-Order Logic}
A set of symbols is called an \emph{alphabet}.
An \emph{arity function} on an alphabet $\sigma$ is a function $\alpha: \sigma \rightarrow \{1,2,\dots\}$.
Given a pair $(\sigma,\alpha)$ consisting of an alphabet $\sigma$ and arity function $\alpha$,
one should think of each symbol $R \in \sigma$ as an $\alpha(R)$-ary relation.
Broader mathematical logic allows for \emph{function} and \emph{constant} symbols as well as relation symbols,
but the axioms for graphs only require relation symbols,
and no expressive power is lost by restricting to relations \cite[Chapter VIII, Theorem 1.3]{ebbinghaus_mathlogic}.

% A \emph{relation symbol of arity $n$} is a symbol $R$ that can take $n$ inputs.
% A \emph{(relational) symbol set} is a collection of relation symbols.

% Symbol sets in broader mathematical logic are allowed to have constant and function symbols as well.
% However, the only structures we consider in this paper are graphs which only require relation symbols.

\begin{defn}\label{defn: first-order}
    Define $\mathcal{V}$ to be the alphabet $\{\lnot,\land,\lor,),(,\exists,\forall,\Rightarrow,\Leftrightarrow\} \cup \{v_i : i \in \mathbb{N}\}$
    and let $\sigma$ be an alphabet with arity function $\alpha$ such that $\sigma \cap \mathcal{V} = \emptyset$.
    We inductively define \emph{first order $\sigma$-formulae}
    as strings obtainable from a finite application of the following rules
    \begin{itemize}
        \item for $i,j \in \mathbb{N}$, $v_i = v_j$ is a first order $\sigma$-formula.
        \item If $R \in \sigma$ satisfies $\alpha(R) = k$ and $i_1,\dots,i_k \in \mathbb{N}$, then $R(v_{i_1}, \dots, v_{i_k})$ is a first order $\sigma$-formula.
        \item If $\phi_1, \phi_2$ are first order $\sigma$-formulae, then $(\phi_1 \lor \phi_2)$, $(\phi_1 \land \phi_2)$, $(\lnot \phi_1)$, $(\phi_1 \Rightarrow \phi_2)$, and $(\phi_1 \Leftrightarrow \phi_2)$ are first order $\sigma$-formulae. 
        \item If $\phi$ is a first order $\sigma$-formula, then $\exists v_i\, \phi_1$ and $\forall v_i \, \phi$ are first order $\sigma$-formulae.
    \end{itemize}
\end{defn}

One can build on first-order logic to obtain more expressive logical systems.
For example, \emph{second order logic} allows for quantification over relations in addition to quantification over variables.
However, gains in expressive power come at the cost of powerful model-theoretic results~\cite{ebbinghaus_mathlogic}.
\emph{Monadic second order logic} is a logic that is more expressive than first-order logic, but less expressive than second order logic.
It is particularly important for graph theory and related algorithmic complexity questions due to a theorem of Courcelle~\cite{Courcelle_graph}
which says that graph properties expressible in this logic can be decided in linear time if one restricts to graphs of bounded treewidth.
In this paper, we only consider first order logic.

\begin{defn}
    Given a first order $\sigma$-formula $\phi$, we inductively define the \emph{free set of $\phi$}, denoted by $\free(\phi)$, as follows
    \begin{itemize}
        \item $\free(v_i = v_j) = \{v_i, v_j\}$
        \item $\free(R(v_{i_1}, \dots, v_{i_k})) = \{v_{i_1}, \dots, v_{i_k}\}$
        \item $\free(\lnot \phi) = \free(\phi)$
        \item if $\phi_1, \phi_2$ are formulae and $* \in\{ \land, \ \lor,  \ \Rightarrow, \, \Leftrightarrow\}$, then $\free(\phi_1 \ * \ \phi_2) = \free(\phi_1) \cup \free(\phi_2)$
        \item $\free(\forall v \ \phi) = \free(\phi)\setminus\{v\}$
        \item $\free(\exists v \ \phi) = \free(\phi)\setminus\{v\}$
    \end{itemize}
    If $\free(\phi) = \emptyset$ then $\phi$ is a \emph{first order $\sigma$-sentence}.
\end{defn}

Given a symbol set $\sigma$, a \emph{$\sigma$-structure}
is a pair $\mathfrak{A} = ( A, \mathfrak{R})$ where $A$ is a set
and $\mathfrak{R}$ is a set indexed by $\sigma$ where the element corresponding to an element $R \in \sigma$ of arity $n$
is a subset $S_R \subseteq A^n$.
If $\sigma$ is empty a $\sigma$-structure is a set.
If $\sigma$ consists of a single unary relation $R$,
then an example of a $\sigma$-structure is a pair of sets $(X, S)$ where $S$ is a distinguished subset of $X$.

If the underlying set of a structure is finite, then the structure is said to be a \emph{finite $\sigma$-structure}. We will be only considering finite structures in this paper. The set of all finite $\sigma$-structures is denoted ${\rm STRUCT}[\sigma]$.

Given a $\sigma$-sentence $\phi$ and a $\sigma$-structure $\mathfrak{A}$,
we say that \emph{$\mathfrak{A}$} is a \emph{model of $\phi$} if $\phi$ is true in $\mathfrak{A}$.
We denote this in symbols as $\mathfrak{A} \models \phi$.
This notation extends set-wise -- if $\mathcal{T}$ is a $\sigma$-theory, we say that $\mathfrak{A}$ is a \emph{model of $\mathcal{T}$} and write $\mathfrak{A} \models \mathcal{T}$ if $\mathfrak{A} \models \phi$ for every $\phi \in \mathcal{T}$.

\begin{ex}\label{ex: theory of graphs}
    Let $\sigma = \{\sim\}$ where $\sim$ is a binary relation symbol.
    The first-order theory of graphs $\mathcal{T}_{\text{Graph}}$ consists of the single sentence $\forall u\, \forall v\, (u \sim v \Rightarrow v \sim u)$.
    A $\sigma$-structure is a pair $\langle V, \{\sim\}\rangle$ where $\sim$ is a binary relation on $V$.
    A $\sigma$-structure $\mathfrak{A}$ such that $\mathfrak{A} \models \mathcal{T}_{\text{Graph}}$ is a graph --- the edges are the pairs $x,y \in V$ such that $x \sim y$.
\end{ex}

We are interested in saying things about models of $\mathcal{T}_{\text{Graph}}$, i.e.~Graphs, but
the theorem we want to use from mathematical logic applies to $\sigma$-structures.
The following construction allows us to bridge this gap.

\begin{defn}\label{defn: Gaifman graph}
    Given a $\{\sim\}$-structure $\mathfrak{A} = (A,E)$, the associated \emph{Gaifman graph}, denoted $G(\mathfrak{A})$,
    is the graph with vertex set $A$ where $uv$ is an edge whenever $(u,v) \in E$ or $(v,u) \in E$.
\end{defn}

When $\mathfrak{A} \models \mathcal{T}_{\rm Graph}$, then $\mathfrak{A}$ and its Gaifman graph $G(\mathfrak{A})$ are essentially the same thing.
In this case we make no distinction between the two objects.

\subsection{Queries}\label{subsection: queries}
Given a symbol set $\sigma$, a \emph{Boolean query on $\sigma$-structures} is a function
\[
    Q: {\rm STRUCT}[\sigma] \rightarrow \{{\rm True},{\rm False}\}.
\]
We are particularly interested in the Boolean queries ${\rm LOC}$ and ${\rm GLO}$ on $\{\sim\}$-structures
which we define as follows
\begin{align}
    {\rm LOC}_d(\mathfrak{A}) &= {\rm True} \quad \Longleftrightarrow \quad \mathfrak{A} \models \mathcal{T}_{\rm Graph} \ {\rm and} \ \mathfrak{A} \ {\rm is \ locally \ } d{\rm -rigid}\label{loc} \\
    {\rm GLO}_d(\mathfrak{A}) &= {\rm True} \quad \Longleftrightarrow \quad \mathfrak{A} \models \mathcal{T}_{\rm Graph} \ {\rm and} \ \mathfrak{A} \ {\rm is \ globally \ } d{\rm -rigid}.\label{glo}
\end{align}

A Boolean query $Q$ is \emph{definable in first-order logic} if there is a first-order sentence $\phi$ such that $Q(\mathfrak{A}) = \text{True}$ if and only if $\mathfrak{A} \models \phi$.
We will show that ${\rm LOC}_d$ and ${\rm GLO}_d$ are not definable in the first order logic by showing that they are not \emph{Hanf local},
which is a property of queries known to be satisfied by all those definable in the first-order logic.
For a more comprehensive introduction to Hanf locality, see \cite[Section 4.1]{libkin_modeltheory}.

\begin{defn}
    A Boolean query $Q$ for $\{\sim\}$-structures is \emph{Hanf-local} if there exists a number $r \geq 0$ such that for any
    $\{\sim\}$-structures $\mathfrak{A}_1,\mathfrak{A}_2$,
    \begin{equation*}
        G(\mathfrak{A}_1) \rightleftarrows_r G(\mathfrak{A}_2) \qquad \text{implies} \qquad Q(\mathfrak{A}_1) = Q(\mathfrak{A}_2)
    \end{equation*}
\end{defn}

\begin{thm}[{\cite[Theorem 3.2]{LIBKIN2006115}}]\label{thm: FO implies Hanf}
    Every first-order definable query is Hanf-local.
\end{thm}

% \begin{prop}[{\cite[Section 4.1]{libkin_modeltheory}}]
%     Graph connectivity cannot be described in first-order logic.
% \end{prop}
% \begin{proof}
%     Suppose that the graph connectivity query $Q$ is Hanf-local and let $hlr(Q) = d$. Let $m > 2d+1$ and choose two graphs $G^1_m$ and $G^2_m$ such that $G^1_m$ is comprised of two disjoint cycles of length $m$, and $G^2_m$ is one cycle of length $2m$. Let $f$ be any bijection between their nodes. The $d-$neighborhood of any node $a$ will be a chain of length $2d$ with $a$ in the middle. Therefore, $G^1_m \rightleftarrows_d G^2_m$ and because we assume the graph connectivity query is Hanf-local, they both must be connected. However, $G^1_m$ is not while $G^2_m$ is. Thus, graph connectivity is not Hanf-local, and so graph connectivity cannot be defined in first-order logic.
% \end{proof}

\section{The construction}\label{sec: the construction}
In order to use Hanf locality to prove that local and global $d$-rigidity are not expressible in the first order logic in the language of graphs,
it suffices to construct for each $d$ and each $r$, two graphs $G$ and $H$ such that $G$ is locally/globally rigid and $H$ is not, along with a bijection between the $r$-neighborhoods of $G$ and $H$.
One of the main results of recent work by Cruickshank, Jackson, and Tanigawa~\cite{cruickshank2024global} says that the one-skeletons of certain simplicial complexes are globally rigid; we will use this result to prove that our construction works.

A \emph{simplicial complex} is a pair $(V,\mathcal{F})$ consisting of a finite set $V$, and a collection $\mathcal{F}$ of subsets of $V$ such that $\emptyset \in \mathcal{F}$ and
$F \in \mathcal{F}$ whenever $F \subset G$ for some $G \in \mathcal{F}$.
Given a simplicial complex $\mathcal{C} = (V,\mathcal{F})$,
$V$ is called the \emph{ground set} or \emph{vertex set} of $\mathcal{C}$
and elements of $\mathcal{F}$ are called \emph{faces}.
The \emph{graph} of a simplicial complex $\mathcal{C} = (V,\mathcal{F})$ is the graph $G(\mathcal{C})$ on vertex set $V$
where $\{u,v\}$ is an edge if and only if it is a face of $\mathcal{C}$.
We will be particularly interested in the rigidity properties of the graphs of the following family of simplicial complexes.

\begin{defn}\label{defn: simplicial complex C^d_n}
    Let $d,n$ be positive integers with $d < n$.
    Let $\mathcal{C}_n^d$ denote the pure simplicial $(d-1)$-complex on ground set $\{0,\dots,n-1\}$
    whose $(d-1)$-faces are all $d$-element subsets of $\{0,\dots,n-1\}$ of the following form for $i = 0,\dots,n-1$
    \[
        \{i,i+1 \ ({\rm mod} \ n),\dots,i+d-1 \ ({\rm mod} \ n)\}.
    \]
\end{defn}

Given a graph $G = (V, E)$, a vertex $v \in V$, and a positive integer $r$, define the \emph{$r$-neighborhood of $v$}, denoted $N^{G}_r(v)$, to be the graph induced by the vertices that are distance at most $r$ away from $v$.
Two graphs $G_1 = (V_1,E_1)$ and $G_2 = (V_2,E_2)$ are \emph{isomorphic} if there exists a bijection $f: V_1 \rightarrow V_2$
such that $uv \in E_1$ if and only if $f(u)f(v) \in E_2$.
This is denoted symbolically as $G_1 \cong G_2$.

We are interested in $r$-neighborhoods of $G(\mathcal{C}_n^d)$.
To this end, we introduce the following related family of simplicial complexes.
They are a non-cyclic analogue of $\mathcal{C}_n^d$.

\begin{defn}
    Let $d,n$ be positive integers with $d < n$.
    Let $\mathcal{P}_n^d$ denote the pure simplicial $(d-1)$-complex on ground set $\{0,\dots,n-1\}$
    whose $(d-1)$-faces are all $d$-element subsets of $\{0,\dots,n-1\}$ of the following form for $i = 0,\dots,n-d$
    \[
        \{i,i+1,\dots,i+d-1\}.
    \]
\end{defn}

\begin{lemma}\label{lemma: isomorphic neighborhoods}
    If $n \ge 2rd+2$ then every $r$-neighborhood of a vertex in $G(\mathcal{C}_n^d)$ is isomorphic to $G(\mathcal{P}_{2rd+1}^d)$.
\end{lemma}
\begin{proof}
    The map $i \mapsto i+1 \ ({\rm mod}) \ n$ on the vertices of $G(\mathcal{C}_n^d)$ is a graph isomorphism,
    and the subgroup of the automorphism group that it generates acts transitively on the vertices of $G(\mathcal{C}_n^d)$.
    Thus it suffices to show that the the $r$-neighborhood of $rd$ is isomorphic to $G(\mathcal{P}_{2rd+1}^d)$.
    Indeed, $i$ is in the $r$-neighborhood of $rd$ in $G(\mathcal{C}_n^d)$ if and only if $0 \le i \le 2rd$.
    The subgraph of $G(\mathcal{C}_{n}^d)$ induced on these vertices is $G(\mathcal{P}_{2rd+1}^d)$.
\end{proof}

The following condition on a pair of graphs can be viewed as a sort of local isomorphism.
We will need it to state the definition of Hanf locality.

\begin{defn}
    If $G_1 = (V_1, E_2)$ and $G_2 = (V_2, E_2)$ are graphs, then we write 
    \begin{equation*}
        G_1 \rightleftarrows_r G_2
    \end{equation*}
    if there is a bijection $f : V_1 \to V_2$ such that for every $v \in V_1$
    \begin{equation*}
        N_r^{G_1}(v) \cong N_r^{G_2}(f(v)).
    \end{equation*}
\end{defn}

\begin{lemma}\label{lemma: bijection between neighborhoods}
    Let $G$ be the disjoint union of two copies of $G(\mathcal{C}_n^d)$.
    If $n \ge 2rd+2$ then $G(\mathcal{C}_{2n}^d) \rightleftarrows_r G$.
\end{lemma}
\begin{proof}
    The vertex set $V$ of $G$ is the disjoint union of two copies of $\{0,\dots,n-1\}$.
    Define $f: V \rightarrow \{0,\dots,2n-1\}$ to be the map that sends the first copy of $\{0,\dots,n-1\}$
    $\{0,\dots,n-1\}$, and the second copy to $\{n,\dots,2n-1\}$.
    The result now follows by Lemma~\ref{lemma: isomorphic neighborhoods}.
\end{proof}

A \emph{$k$-face} of a simplicial complex $(V,\mathcal{F})$ is a face of cardinality $k+1$\footnote{The reason that a $k$-face is defined to have $k+1$ vertices is because a $k$-face is an abstraction of a $k$-dimensional simplex, which is defined to be the convex hull of $k+1$ affinely independent points in a Euclidean space}.
A \emph{simplicial $k$-complex} is a simplicial complex whose faces of maximum cardinality are $k$-faces.
A simplicial $k$-complex is \emph{pure} if every face of maximum cardinality is a $k$-face.

\begin{defn}
    A pure simplicial $k$-complex is called a \emph{$k$-circuit} if every $(k-1)$ face is contained in an even number of $k$ faces.
\end{defn}

Our interest in $k$-circuits is due to the following result of Cruickshank, Jackson, and Tanigawa.
They were primarily interested in the special case of triangulated manifolds.

\begin{thm}[{\cite[Theorem 1.4]{cruickshank2024global}}]\label{thm: global rigidity of k circuits}
    Let $k \ge 3$ and let $G$ be the graph of a simplicial $k$-circuit.
    Then $G$ is globally $(k+1)$-rigid if and only if $G = K_{k+1}$ or $G = K_{k+2}$ or $G$ is $(k+2)$-connected. 
\end{thm}

The following lemma uses Theorem~\ref{thm: global rigidity of k circuits} to prove that certain graphs of the form $G(\mathcal{C}_n^d)$ are globally $d$-rigid.

\begin{lemma}\label{lemma: globally rigid}
    Let $n,d$ be integers with $n \ge 2d+2$.
    Then $G(\mathcal{C}_n^d)$ is globally rigid in~$\mathbb{R}^d$.
\end{lemma}
\begin{proof} 
    By Theorem~\ref{thm: global rigidity of k circuits}, it suffices to show that $\mathcal{C}_n^d$ is a simplicial $(d-1)$-circuit that is $(d+1)$-connected.
    It is indeed a $(d-1)$ circuit since each $(d-2)$-face $\{i,i+1 \ ({\rm mod} \ n),\dots,i+d-2 \ ({\rm mod} \ n)\}$ is contained in exactly two $(d-1)$-faces, namely
    \[
        \{i,i+1 \ ({\rm mod} \ n),\dots,i+d-1 \ ({\rm mod} \ n)\}
    \]
    and
    \[
        \{i-1,i,i+1 \ ({\rm mod} \ n),\dots,i+d-2 \ ({\rm mod} \ n)\}.
    \]

    We now argue that $G(\mathcal{C}^d_n)$ is $(d+1)$-connected.
    Let $S = \{v_1,\dots,v_d\}$ be a set of $d$ vertices, and consider the graph $H$ obtained by removing $S$.
    We argue that $H$ is connected.
    If $S$ is a $(d-1)$-face of $\mathcal{C}^n_d$ then $H$ is isomorphic to $G(\mathcal{P}^d_{n-d})$ which is connected.
    If there exists a $(d-1)$ face $F$ of $\mathcal{C}^n_d$ such that $F \setminus S = \{v\}$, let $H'$ denote
    the graph obtained from $G(\mathcal{C}^d_n)$ by removing $S \cup \{v\}$.
    Then $H'$ is a one-vertex deletion of a graph isomorphic to the 2-connected graph $G(\mathcal{P}^d_{n-d})$
    so $H'$ is connected.
    Since $H'$ is a subgraph of $H$, this implies $H$ is connected.

    Finally, assume every $(d-1)$-face of $\mathcal{C}^d_n$ has at least two vertices not in $S$
    and let $u,w$ be two vertices of $H$.
    Let $u_1,\dots,u_k$ be a path from $u$ to $w$ in $G(\mathcal{C}^d_n)$.
    If there exists an $i$ such that $u_i \in S$, then $i \in \{2,\dots,k-1\}$ as $u_1 = u$ and $u_k = w$.
    By induction on the number of $u_j$'s in $S$,
    it suffices to show how to replace $u_i$ with a path from $u_{i-1}$ to $u_{i+1}$ that avoids elements of $S$.
    
    Without loss of generality, assume $u_{i-1} \le u_i \le u_{i+1}$.
    Let $F_-$ and $F_+$ denote the following $(d-1)$-faces of $\mathcal{C}^d_n$
    \begin{align*}
        F_- &:= \{u_{i-1},u_{i-1} + 1,\dots, u_{i-1} + d-1\} \quad {\rm and} \\
        F_+ &:= \{u_{i+1}-(d-1),\dots,u_{i+1}-1 ,u_{i+1}\}.
    \end{align*}
    Since $u_i$ is adjacent to both $u_{i-1}$ and $u_{i+1}$ in $G(\mathcal{C}^d_n)$,
    $F_-$ and $F_+$ both contain $u_i$.
    By our assumption, $F_-$ contains a vertex $w_- \neq u_i$ with $w_- \notin S$ and $F_+$ contains a vertex $w_+ \neq u_i$ with $w_+ \notin S$.
    Then $u_{i-1}$ is adjacent to $w_-$ and $u_{i+1}$ is adjacent to $w_+$ in $G(\mathcal{C}^d_n)$.
    If $\mathcal{C}^d_n$ has a face containing both $w_-$ and $w_+$
    then $w_-$ and $w_+$ are adjacent in $G(\mathcal{C}^d_n)$ thus giving us a path from $u_{i-1}$ to $u_{i+1}$ avoiding $S$.
    If not, then we can repeat this process with $w_-$ in place of $u_{i-1}$ and $w_+$ in place of $u_{i+1}$.
    This completes the proof by induction on $u_{i+1}-u_{i-1}$.
\end{proof}

We are now ready to prove our main result.

\begin{thm}\label{thm: not first-order}
    Local/global rigidity is not definable in first-order theory of graphs.
\end{thm}
\begin{proof}
    By Theorem~\ref{thm: FO implies Hanf}, it suffices to show that the queries ${\rm LOC}_d$ and ${\rm GLO}_d$, as defined in \eqref{loc} and \eqref{glo}, are not Hanf local.
    Indeed, let $r \ge 0$ and define $G_1:= G(\mathcal{C}_{4rd+4}^d)$ and define $G_2$ to be the disjoint union of two copies of $G(\mathcal{C}_{2rd+2}^d)$.
    Lemma~\ref{lemma: bijection between neighborhoods} implies that $G_1 \rightleftarrows_r G_2$
    and Lemma~\ref{lemma: globally rigid} implies that $G_1$ is globally $d$-rigid and therefore locally $d$-rigid as well.
    Since $G_2$ is disconnected, it is neither locally nor globally $d$-rigid.
    Thus ${\rm LOC}_d(G_1) = {\rm GLO}_d(G_1) = {\rm True}$ whereas ${\rm LOC}_d(G_2) = {\rm GLO}_d(G_2) = {\rm False}$
    and so neither ${\rm LOC}_d$ nor ${\rm GLO}_d$ is Hanf local.
\end{proof}

\bibliographystyle{abbrv}
\bibliography{bibliography}

@book{libkin_modeltheory,
  author = {Leonid Libkin},
  year = {2012},
  title = {Elements of Finite Model Theory},
  publisher = {Springer}
}

@article{LIBKIN2006115,
title = {Locality of Queries and Transformations},
journal = {Electronic Notes in Theoretical Computer Science},
volume = {143},
pages = {115-127},
year = {2006},
note = {Proceedings of the 12th Workshop on Logic, Language, Information and Computation (WoLLIC 2005)},
issn = {1571-0661},
doi = {https://doi.org/10.1016/j.entcs.2005.04.041},
url = {https://www.sciencedirect.com/science/article/pii/S157106610505231X},
author = {Leonid Libkin}
}

@book{ebbinghaus_mathlogic,
    author = {Heinz-Dieter Ebbinghaus and J\"org Flum and and Wolfgang Thomas},
    title = {Mathematical Logic, Third Ed.},
    publisher = {Springer},
    year = {2021}
}

@article{clinch2019abstract,
  title={Abstract 3-Rigidity and Bivariate $C_2^1$-Splines II: Combinatorial Characterization},
  author={Clinch, Katie and Jackson, Bill and Tanigawa, Shin-ichi},
  journal={arXiv preprint arXiv:1911.00207},
  year={2019}
}

@book{Courcelle_graph,
    author = {Bruno Courcelle and Joost Engelfriet},
    title = {Graph structure and monadic second-order logic. A
 language-theoretic approach},
    publisher = { Cambridge University Press},
    year = {2012}
}

@article{cruickshank2024global,
  title={Global rigidity of triangulated manifolds},
  author={Cruickshank, James and Jackson, Bill and Tanigawa, Shin-Ichi},
  journal={Advances in Mathematics},
  volume={458},
  pages={109953},
  year={2024},
  publisher={Elsevier}
}

@article{asimow1978rigidity,
  title={The rigidity of graphs},
  author={Asimow, Leonard and Roth, Ben},
  journal={Transactions of the American Mathematical Society},
  volume={245},
  pages={279--289},
  year={1978}
}

@article{gortler2010characterizing,
  title={Characterizing generic global rigidity},
  author={Gortler, Steven J and Healy, Alexander D and Thurston, Dylan P},
  journal={American Journal of Mathematics},
  volume={132},
  number={4},
  pages={897--939},
  year={2010},
  publisher={Johns Hopkins University Press}
}

@article{pollaczek1927gliederung,
  title={{\"U}ber die gliederung ebener fachwerke},
  author={Pollaczek-Geiringer, Hilda},
  journal={ZAMM-Journal of Applied Mathematics and Mechanics/Zeitschrift f{\"u}r Angewandte Mathematik und Mechanik},
  volume={7},
  number={1},
  pages={58--72},
  year={1927},
  publisher={Wiley Online Library}
}

@article{connelly2005generic,
  title={Generic global rigidity},
  author={Connelly, Robert},
  journal={Discrete \& Computational Geometry},
  volume={33},
  number={4},
  pages={549--563},
  year={2005},
  publisher={Springer}
}

@inproceedings{jackson2007notes,
  title={Notes on the rigidity of graphs},
  author={Jackson, Bill},
  booktitle={Levico Conference Notes},
  volume={4},
  year={2007}
}

\end{document}